\documentclass[11pt]{amsart}

\usepackage{amsmath}
\usepackage{amsthm}
\usepackage{amsfonts}
\usepackage{comment}
\usepackage{amssymb,amsrefs}
\usepackage{hyperref}
\usepackage{mathtools}
\usepackage{bm}
\usepackage[margin=1in]{geometry}
\usepackage{mathrsfs}
\usepackage{enumerate}
\usepackage{enumitem}
\usepackage{todonotes}
\usepackage{enumerate}
\newtheorem{theorem}{Theorem}[section]

\newtheorem{lemma}[theorem]{Lemma}

\newtheorem{conjecture}[theorem]{Conjecture}

\theoremstyle{definition}

\newtheorem*{theorem*}{Theorem}
\newtheorem*{proposition*}{Proposition}
\newtheorem*{lemma*}{Lemma}

\theoremstyle{remark}
\newtheorem*{remark}{Remark}

\newtheorem*{note}{Note}

\numberwithin{equation}{section}

\def\al{{\alpha}}
\def\be{{\beta}}

\def\eps{{\epsilon}}

\def\ld{{\lambda}}

\def\HH{{\mathbb H}}

\def\Ak1{{\mathbb A_k ^1}}

\def\RR{{\mathbb R}}

\def\ZZ{{\mathbb Z}}
\makeatletter
\newcommand*{\defeq}{\mathrel{\rlap{%
                     \raisebox{0.3ex}{$\m@th\cdot$}}%
                     \raisebox{-0.3ex}{$\m@th\cdot$}}%
                     =}
\makeatother

\def\cal{\mathcal}

\def\cC{{\cal C}}

\def\cS{{\mathfrak S}}

\def\SL{\operatorname{SL}}

\def\del{\delta}

\def\geom{\operatorname{geom}}

\def\PSL{\operatorname{PSL}}

\newcommand{\norm}[1]{\left\lVert#1\right\rVert}

\newcommand*{\shom}{\mathscr{H}\kern -.5pt om}
\allowdisplaybreaks
\begin{document}

\title[Probability laws for the distribution of geometric lengths]{Probability laws for the distribution of geometric lengths when sampling by a random walk in a Fuchsian fundamental group}

\date{\today}

\author[Peter S. Park]{Peter S. Park}
\address{Department of Mathematics, Harvard University, Cambridge, MA 02138}
\email{\href{mailto:pspark@math.harvard.edu}{{\tt pspark@math.harvard.edu}}}

\maketitle

\begin{abstract}
Let $S=\Gamma\backslash \mathbb{H}$ be a hyperbolic surface of finite topological type, such that the Fuchsian group $\Gamma \le \operatorname{PSL}_2(\mathbb{R})$ is non-elementary, and consider any generating set $\mathfrak S$ of $\Gamma$. When sampling by an $n$-step random walk in $\pi_1(S) \cong \Gamma$ with each step given by an element in $\mathfrak S$, the subset of this sampled set comprised of hyperbolic elements approaches full measure as $n\to \infty$, and for this subset, the distribution of geometric lengths obeys a Law of Large Numbers, Central Limit Theorem, Large Deviations Principle, and Local Limit Theorem. We give a proof of this known theorem using Gromov's theorem on translation lengths of Gromov-hyperbolic groups.
 \end{abstract}
 
\section{Introduction}\label{sec:intro}
Let $S=\Gamma\backslash \HH$ be a hyperbolic surface of finite topological type, where $\Gamma \le \PSL_2(\RR)$ is a Fuchsian group that acts on $\HH$ by fractional linear transformations. By the assumption of finite topological type, the fundamental group $\pi_1(S)\cong \Gamma$ is finitely presented, and in particular, finitely generated. Fix a finite \emph{spanning set} $\cS$ of $\Gamma$, i.e., a subset whose multiplicative closure is equal to all of $\Gamma$. Then, $g\in \Gamma$ has a  notion of \emph{algebraic length}, defined by 
\[
\ell_\cS(g) \defeq \inf_{
\substack{k \in \ZZ_{\ge 0} : \exists e_1, \ldots, e_k \in \cS\\  g = e_1 \cdots e_k } } k.
\]
If $g \in \Gamma$ is hyperbolic, then it also has a notion of \emph{geometric length} defined as follows. In the bijective correspondence between conjugacy classes of $\pi_1(S)$ and free homotopy classes of loops in $S$, the hyperbolic conjugacy classes precisely correspond to free homotopy classes of loops with a unique  representative that is geodesic with respect to the hyperbolic metric of $S$. This gives the definition of the geometric length $\geom(g)$ of $g$: the length of the geodesic representative of the free homotopy class of loops corresponding to the conjugacy class of $g$. 

In the absence of a straightforward formula for the geometric length of a given word---such as the Pythagorean Theorem for the unit square torus with the standard fundamental group generators---a general question naturally arises:

\hspace{1pt}

\begin{center}
\emph{What can we say about the relationship between}

\emph{the algebraic length and the geometric length?}
\end{center}

\hspace{1pt}

\noindent The simplest case in our setting is when $S$ is a pair of pants, i.e., $S^2$ minus three disjoint open disks endowed with a hyperbolic metric that makes the three boundary components, which we denote by $B_1,B_2,$ and $B_3$, geodesics. The hyperbolic metric can uniquely be described by specifying the geometric lengths of  $B_1,B_2,$ and $B_3$. Note that $\pi_1(S)$ is isomorphic to the free group $F_2$ on two generators, and we can choose the two free generators $X$ and $Y$ to be loops around $B_1$ and $B_2$ respectively, so that $XY$ is a loop around $B_3$. In the case that $\cS=\{X,Y,X^{-1},Y^{-1}\}$, Chas--Li--Maskit~\cite{chasconj} conjectured from computational evidence the following correlative relationship between algebraic and geometric length.

\begin{conjecture}[Chas--Li--Maskit]\label{conj:chas}
Let $S$ be the pair of pants such that $B_1, B_2,$ and $B_3$ have geometric lengths $\ell_1,\ell_2,$ and $\ell_3$. Let $\mu_n$ denote the uniform random distribution on the set $R_n$ of cyclic reduced $\cS$-words of algebraic length $n$. There exist positive constants $\kappa=\kappa(\ell_1,\ell_2,\ell_3)$ and $\nu=\nu(\ell_1,\ell_2,\ell_3)$ such that for any $a<b$, 
\[
\lim_{n\to \infty}\int_{R_n} \chi_{[a,b]} \left(\frac{\geom(w)-\kappa n}{\sqrt{n}}\right)d \mu_{n}(w)=\int_a^b \frac {e^{-\frac{s^2}{2\nu}}}{\sqrt{2\pi\nu}} ds.
\]
\end{conjecture}
Indeed, since all conjugacy classes of a uniformizing Fuchsian group of a pair of pants are hyperbolic, they have a well-defined notion of geometric length. In the above, \emph{$\cS$-words} are symbolic expressions in the elements of $\cS$, \emph{cyclic $\cS$-words} are equivalence classes of $\cS$-words up to cyclic conjugation, and a word or cyclic word is \emph{reduced} if and only if no adjacent pairs of elements are inverses (with the caveat that for cyclic words, the first and last $\cS$-elements are considered adjacent). Cyclic reduced words are useful because they are in bijective correspondence with conjugacy classes; for instance, our earlier work~\cite{park} asymptotically computes the growth of conjugacy classes of commutators in free groups and free products of two finite groups by using this bijective correspondence with cyclic reduced words. However, we note that since the proportion of length-$n$ reduced $\cS$-words that have the maximum possible $n$ cyclic conjugates approaches $1$ as $n\to \infty$, Conjecture~\ref{conj:chas} is equivalent to the statement that an analogous Central Limit Theorem (CLT) type theorem holds for reduced $\cS$-words instead of cyclic reduced $\cS$-words.

Conjecture~\ref{conj:chas} desires a CLT-type theorem for the distribution of a geometric quantity of loops when sampling by a different, but related algebraic quantity. A theorem of this  type was proven by Chas--Lalley~\cite{chasthm}, who proved that for $S$ compact, a CLT-type theorem holds for the distribution of self-intersection numbers of loops when sampling by algebraic length, a phenomenon that also was previously suggested by computational evidence. 

In this paper, we prove a theorem of a similar spirit that demonstrates an analogue of Conjecture~\ref{conj:chas} sampling by \emph{symbolic length} rather than algebraic length. The symbolic length of a $\cS$-word is defined by the number of elements of $\cS$ in the expression, counted with multiplicity. Note that unlike algebraic length, symbolic length is not well-defined on $\Gamma$; it is only well-defined on $\cS$-words, some of which may be equal as group elements in $\Gamma$. However, for reduced $\cS$-words in the free fundamental group of a pair of pants, symbolic length and algebraic length coincide. Thus, the following theorem can also be thought of as an analogue of Conjecture~\ref{conj:chas} regarding \emph{all} words rather than just reduced words.

\begin{theorem}\label{thm:allwords}
Let $S$ be the pair of pants such that $B_1, B_2,$ and $B_3$ have geometric lengths $\ell_1,\ell_2,$ and $\ell_3$. Let $\mu_n$ denote the uniform random distribution on the set $W_n$ of   $\cS$-words of symbolic length $n$. There exist positive constants $\kappa=\kappa(\ell_1,\ell_2,\ell_3)$ and $\nu=\nu(\ell_1,\ell_2,\ell_3)$ such that for any $a<b$, 
\[
\lim_{n\to \infty}\int_{W_n} \chi_{[a,b]} \left(\frac{\geom(w)-\kappa n}{\sqrt{n}}\right)d \mu_{n}(w)=\int_a^b \frac {e^{-\frac{s^2}{2\nu}}}{\sqrt{2\pi\nu}} ds.
\]
\end{theorem}

The above theorem is a special case of the following.

\begin{theorem}\label{thm:main}
Let $S=\Gamma \backslash \HH$ be a hyperbolic surface of finite topological type, such that the Fuchsian group $\Gamma \le \PSL_2(\RR)$ is non-elementary. Then, for any spanning set $\cS$ of $\Gamma$ and  any  probability measure $\mu$ on $\cS$ (with support equal to $\cS$), the $n$th convolution power $\mu^{*n}$ on the set $W_n$ of $\cS$-words of symbolic length $n$ satisfies the following:
\begin{enumerate}
\item Let $H_n \subseteq W_n$ denote the subset comprised of elements that are hyperbolic in $\Gamma$. We have that as $n\to \infty$, the measure of $W_n \setminus H_n$ limits to $0$.

\item There exist positive constants $\kappa=\kappa(\Gamma,\cS,\mu)$ and $\nu=\nu(\Gamma,\cS,\mu)$  such that for any bounded, continuous function $\psi$ on $\RR$, we have
\[
\lim_{n\to \infty}\int_{H_n} \psi \left(\frac{\geom(w)-\kappa n}{\sqrt{n}}\right)d \mu^{*n}(w)=\int_\RR \psi(s) \frac {e^{-\frac{s^2}{2\nu}}}{\sqrt{2\pi\nu}} ds.
\]
\end{enumerate}
\end{theorem}
\noindent Note that conclusion $2$ above is equivalent to the statement that the distribution
\[
\frac{\geom(w)-\kappa n}{\sqrt{n}}
\]
(with law $\mu^{*n}$) converges in distribution to the Gaussian distribution with mean $0$ and variance $\nu$. In particular, this means that in the statement of conclusion $2$, the function $\psi$ can be taken to be the characteristic function of an interval. Thus, in the case that $S$ is a pair of pants---where $\cS$ can be taken to be $\{X,Y,X^{-1},Y^{-1}\}$, $\mu$ can be taken to be
\[
\frac{1}{4}\delta_{X}+\frac{1}{4}\delta_{Y}+\frac{1}{4}\delta_{X^{-1}}+\frac{1}{4}\delta_{Y^{-1}},
\]
and all elements of $\Gamma$ are hyperbolic---taking $\psi=\chi_{[a,b]}$ in Theorem~\ref{thm:main} yields Theorem~\ref{thm:allwords}.

We will prove Theorem~\ref{thm:main} by applying the theory of random walks on certain linear groups---specifically, a non-commutative CLT-type theorem for matrix products arising from a random walk, each of whose possible steps represents multiplying by a matrix corresponding to an element of $\cS$. This theory has been built by Furstenberg--Kesten~\cite{furstenberg}, Le Page~\cite{lepage}, Guivarc\cprime h--Raugi~\cite{guivarch}, Gol\cprime dshe\u\i d--Margulis~\cite{margulis}, and Benoist--Quint~\cite{cltpaper}. Due to their work, we know that if a random walk on a matrix group satisfies certain hypotheses, the logarithms of the operator norms of the resulting matrices satisfy natural probability laws, including not only the CLT, but also the Law of Large Numbers (LLN), the Large Deviations Principle (LDP), and the Local Limit Theorem (LLT). These probability laws are useful for our purposes, since the geometric length of a hyperbolic element of $\Gamma$ is precisely the logarithm of the operator norm of its corresponding $\PSL_2(\RR)$-matrix. 

We further use the CLT and LLN (specifically, the positivity of the first Lyaponuv exponent) in this context, along with a result of Gromov~\cite[Corollary 8.1.D]{gromov} (the positivity of translation lengths on hyperbolic groups), to prove the first statement of Theorem~\ref{thm:main}. This shows that the $n$-step random-walk matrices are hyperbolic with probability (limiting to) $1$ as $n\to \infty$, and thus, the probability laws on the logarithms of the operator norms of these matrices have a geometric interpretation as probability laws on the geometric lengths of their corresponding geodesics. This geometric interpretation for the CLT is precisely the second statement of Theorem~\ref{thm:main}.

\begin{note}
It has come to our attention that the statements of this paper are known---for example, they follow from Corollary 14.16 and Theorem 14.22 of~\cite{randomwalk}---and that the Chas--Li--Maskit Conjecture has recently been proven in~\cite{gtt}. 
\end{note}

\section{Random walks on linear groups}\label{sec:randomwalk}

Let $V\defeq \RR^2$ with a choice of Euclidean norm $|\cdot|$, and let $\norm{\cdot}$ denote the operator norm on $\SL(V)$.  Let $\mu$ be a Borel probability measure on $G\defeq \SL(V)$. Let $A$ denote the support of $\mu$, and $\Gamma_\mu$, the closed sub-semigroup of $G$ spanned by $A$. For nonzero $v\in V$, let $\bar v$ be the line spanned by $v$, and extend the group action of $G$ on $V$ to one on the set of lines in $V$, given by $g\bar v=\overline{gv}$. We say that a group acts \emph{strongly irreducibly} on $V$ if and only if no proper finite union of vector subspaces of $V$ is invariant under that group action. 

Suppose the following  hypotheses hold; note that  hypothesis $(1)$ is not optimal and can be weakened to a finite second moment hypothesis~\cite{cltpaper}, but suffices for our purposes.
\begin{enumerate}
\item $\int_G \norm{g}^\al d\mu(g)<\infty$ for some $\al>0$.
\item $\Gamma_\mu$ is unbounded and acts strongly irreducibly on $V$.
\end{enumerate}
By Jensen's Inequality, hypothesis $(1)$ implies that the first moment 
\[
\int_G \log \norm{g} d\mu(g) 
\]
is finite; accordingly, it follows from the submultiplicativity of $\norm{\cdot}$ that the first Lyapunov exponent 
\[
\ld_1\defeq \lim_{n \to \infty} \frac{1}{n}\int_G\log \norm{g} d\mu^{*n}(g)
\]
is finite. In the above, the $n$th convolution power $\mu^{*n}$ is a measure corresponding to the distribution of  $g=g_n\cdots g_1$ for i.i.d. random matrices $g_1,\ldots, g_n$ in $G$ chosen according to law $\mu$.

Furthermore, define the one-sided Bernoulli space $B \defeq A^{\ZZ_{>0}} \defeq\{(g_1,g_2,\ldots) : g_n \in A \text{ for all $n$}\}$, endowed with the Bernoulli probability measure $\be \defeq \mu^{\otimes \ZZ_{>0}}$.  Then, we have the following LLN-type theorem due to Furstenberg.
\begin{theorem}[{{\cite[Theorem 1.6]{randomwalk}}}]\label{thm:prob1}
Suppose  hypotheses $(1)$ and $(2)$ hold. For  $\be$-almost all $b\in B$, we have
\[
\lim_{n\to \infty }\frac{1}{n}\log \norm{g_n\cdots g_1}=\ld_1,
\]
and furthermore, $\ld_1>0$.
\end{theorem}

We also have the following CLT-type theorem for  $\log \norm{g}$.
\begin{theorem}[{{\cite[Theorem 1.7]{randomwalk}}}]\label{thm:prob2}
Suppose  hypotheses $(1)$ and $(2)$ hold. The limit
\[
\Phi\defeq \lim_{n\to \infty}\int_G (\log \norm{g}-\ld_1 n)^2 d\mu^{*n}(g)
\]
exists and is positive. For any bounded, continuous function $\psi$ on $\RR$, we have
\[
\lim_{n\to \infty}\int_G \psi\left(\frac{\log\norm{g}-\ld_1 n}{\sqrt{n}}\right)d \mu^{*n}(g)=\int_\RR \psi(s)\frac {e^{-\frac{s^2}{2\Phi}}}{\sqrt{2\pi\Phi}} ds.
\]
Equivalently, the distribution 
\[\frac{\log\norm{g}-\ld_1 n}{\sqrt{n}}
\]
(with law $\mu^{*n}$) converges in distribution to the Gaussian distribution with mean $0$ and variance $\Phi$. 
\end{theorem}

Moreover, we have the following LDP-type theorem.

\begin{theorem}[{{\cite[Theorem 1.9]{randomwalk}}}]\label{thm:prob4}
For any $t_0>0$, we have
\[
\limsup_{n\to \infty}\mu^{*n}(\{g\in G: |\log \norm{g}-\ld_1n| \ge nt_0 \})^{\frac{1}{n}} <1.
\]
\end{theorem}

Finally, we have the following LLT-type theorem.

\begin{theorem}[{{\cite[Theorem 1.10]{randomwalk}}}]\label{thm:prob5}
For any $a_1<a_2$, we have
\[
\lim_{n\to \infty}\sqrt{n} \mu^{*n}( \{ g\in G : \log \norm{g} -\ld_1 n \in [a_1, a_2]\})=\frac{a_2-a_1}{\sqrt{2\pi \Phi}}.
\]
\end{theorem}

\begin{remark}
The survey \textit{Random walks on reductive groups}~\cite{randomwalk} by Benoist--Quint initially gives Theorems~\ref{thm:prob1} to \ref{thm:prob5} as statements about the random-walk distribution of $\log |gv|$ for an arbitrary $v\in V\setminus \{0\}$. However, the analogous probability laws  for  $\log \norm{g}$ can be easily deduced from those for $\log |gv|$ by a renormalization, as stated in~\cite[p. 16]{randomwalk}. 
\end{remark}

\section{Proof of Theorem~\ref{thm:main}}\label{sec:mainproof}

In our setting,  $\Gamma\backslash \HH$ has no orbifold points by assumption, so $\Gamma$ is torsion-free. Furthermore, since $\Gamma$ is non-elementary, it contains a free subgroup $F \cong F_2$ comprised entirely of hyperbolic matrices (see, for instance,~\cite[Proposition 3.1.2]{hubbard}). Let $X,Y \in \SL_2(\RR)$ be two  matrices such that $\overline{X}$ and $\overline{Y}$ (where placing a bar above an $\SL_2(\RR)$ matrix denotes its $\PSL_2(\RR)$ equivalence class) freely generate $F$. Suppose $\cS = \{\overline{Z_1},\ldots, \overline {Z_k}\}$, where $\overline{Z_1},\ldots, \overline {Z_k}$ are all  distinct. Correspondingly, let $A=\{Z_1,\ldots,Z_k\}$.  We wish to prove probability laws for the distribution of geometric lengths from $H_n$, when sampling by a random walk with law given by a probability measure
\[
\sum_{j=1}^k  c_j \del_{Z_j}
\]
for arbitrary positive constants  $ c_1,\ldots, c_k$  that add to $1$. In order to apply the theorems introduced in Section~\ref{sec:randomwalk}, we need to show that $H_n \subseteq W_n$ approaches full measure as $n\to \infty$, as well as verify hypotheses $(1)$ and $(2)$ for our setting.

\begin{lemma}\label{lem:hyp}
Hypotheses $(1)$ and $(2)$ hold for $\mu$.
\end{lemma}

\begin{proof}
Hypothesis $(1)$ is clear. The first part of hypothesis $(2)$ is clear; indeed, $X$ is hyperbolic, and thus $X^n$ is unbounded as $n\to \infty$. Furthermore, $\Gamma_\mu$ acts strongly irreducibly on $V$. Indeed, because $X$ and $Y$ do not commute, the major and minor axes of $X$ and $Y$ correspond to four distinct lines in $V$. Thus, given a line  $\ell$ in $V$, without loss of generality, we can assume that $\ell$ is not equal to either the major or the minor axis of $X$ (of $Y$, for the other case). Then, for any set $L$ of finitely many lines in $V$, there exists sufficiently large $n$ so that $X^n \ell \notin L$. 
\end{proof}

The above lemma proves Theorems~\ref{thm:prob1} to \ref{thm:prob5} for the law $\mu$. In particular, Theorem~\ref{thm:prob2} holds, demonstrating the CLT-type statement regarding $\log \norm{g}$ for $g \in A^n$, with mean $\ld_1$ and variance $\Phi$. However, only conjugacy classes of \emph{hyperbolic} matrices $g$ have a well-defined notion of geometric length, which is then given by $\log \norm{g}$. It still remains to show that  $H_n \subseteq W_n$ approaches full measure as $n\to \infty$. For  the sake of adherence to our most recent notation, let $\underline{H_n}$ denote the subset of $A^n$ comprised of elements that are hyperbolic in $\SL_2(\RR)$.

\begin{lemma}
We have that as $n\to \infty$, the measure of  $A^n \setminus \underline{H_n}$ limits to $0$.
\end{lemma}

\begin{proof}
We need to show that the subset of $g \in A^n$ that are  non-hyperbolic (i.e., parabolic, since $\Gamma$ is torsion-free) as $\SL_2(\RR)$-matrices has measure going to $0$ as $n \to \infty$. There are finitely many primitive parabolic conjugacy classes $\cC_1,\ldots, \cC_m$ in $\Gamma$;  each $\cC_j$ can be $\PSL_2(\RR)$-conjugated to 
\[
\overline{\begin{pmatrix}
1 & t_j \\
0 & 1
\end{pmatrix}}.
\]
A parabolic conjugacy class of $\Gamma$ is precisely a power of one such class $\cC_j$ (with the caveat that the identity conjugacy class is the trivial power of any $\cC_j$). For each $\cC_j$, define
\[
s_j = \inf_{\text{$\cS$-word $w$  contained in some nontrivial power $\cC_j^a$}} \frac{\text{symbolic length of $w$ in $\cS$}}{a}.
\]
Suppose for the sake of a contradiction that $s_j=0$. Then, there would be  primitive $\cS$-words $\{w_i\}_{i\in \ZZ_{> 0}}$, each of which is contained in some nontrivial power   $\cC_j^{a_i}$,  such that
\[
\frac{\ell_\cS(w_i)}{a_i} 
\]
monotonically decreases to $0$, since the algebraic length $\ell_\cS(w_i)$ is upper-bounded by the symbolic length. We then have  that $w_1^{a_i}$ is conjugate to $w_i^{a_1}$ for all $i \in \ZZ_{>0}$. This implies that the \emph{translation length}~\cite[p. 146]{tlength} of $w_1$, defined by 
\[
\liminf_{u\to \infty}\frac{\ell_{\cS}(w_1^u)}{u},
\]
is $0$. However, it is a result of Gromov~\cite[Corollary 8.1.D]{gromov} that the translation length of an infinite-order element of a hyperbolic group is nonzero, and by the \v{S}varc--Milnor Lemma~\cite{svarc,milnor}, any finitely-generated Fuchsian group $\Gamma$ is hyperbolic. This contradiction shows that $s_j>0$. 

Thus, for  $g \in A^n$ that are parabolic as $\SL_2(\RR)$-matrices, say contained in a power of $\cC_j$, we have that $\norm{g}$ is at most 
\[
\norm{\begin{pmatrix}
1 & \frac{n}{s_j}t_j \\
0 & 1
\end{pmatrix}},
\]
using the bound that the exponent of the power of $\cC_j$ in which $g$ is contained is $\le n/s_j$. It is well-known that 
\[
\norm{\begin{pmatrix}
1 & x \\
0 & 1
\end{pmatrix}}
\]
grows like a polynomial in $|x|$ as $|x| \to \infty$. Consequently, for all $n>1$ (so that $\log n > 0$), we have
\[
\log \norm{g} \le C_j \log n
\]
for some constant $C_j > 0$.

Let $\eps>0$ be arbitrary. Fix $\alpha>0$ large enough so that
\[
\int_{-\infty}^{-\alpha}\frac {e^{-\frac{s^2}{2\Phi}}}{\sqrt{2\pi\Phi}} ds \le \frac{\eps}{2}.
\]
Using the positivity of $\ld_1$ given by Theorem~\ref{thm:prob1}, there exists $N_1>0$  so that for all $n>N_1$, 
\[
\frac{\ld_1 n-C_j \log n }{\sqrt n}\ge \alpha.
\]
Next, an analogous discussion to that of Section~\ref{sec:intro} demonstrates that the statement of  Theorem~\ref{thm:prob2} holds for $\psi=\chi_{(-\infty,-\al]}$. Thus, there exists $N_2>0$ such that for all $n>N_2$,
\[
\int_{G} \chi_{(-\infty,-\al]} \left(\frac{\log\norm{g}-\ld_1n}{\sqrt{n}} \right)d \mu^{*n}(g)
\]
is within $\eps/2$ of 
\[
\int_{-\infty}^{-\alpha}\frac {e^{-\frac{s^2}{2\Phi}}}{\sqrt{2\pi\Phi}} ds.
\]
Applying the triangle inequality, we conclude that for $n>\max(N_1,N_2)$, 
\begin{align*}
&\int_{G}\chi_{\left(-\displaystyle\infty, \displaystyle\frac{C_j \log n-\ld_1n }{\sqrt n}\right]}\left(\frac{\log\norm{g}-\ld_1n}{\sqrt{n}}\right) d \mu^{*n}(g) \\
\le &\int_{G}\chi_{(-\infty, -\al]}\left(\frac{\log\norm{g}-\ld_1n}{\sqrt{n}}\right) d \mu^{*n}(g) \\
\le & \left|\int_{G}\chi_{(-\infty, -\al]}\left(\frac{\log\norm{g}-\ld_1n}{\sqrt{n}}\right) d \mu^{*n}(g)-\int_{-\infty}^{-\alpha}\frac {e^{-\frac{s^2}{2\Phi}}}{\sqrt{2\pi\Phi}} ds\right|+\int_{-\infty}^{-\alpha}\frac {e^{-\frac{s^2}{2\Phi}}}{\sqrt{2\pi\Phi}} ds \\
\le &\frac{\eps}{2}+\frac{\eps}{2} = \eps.
\end{align*}
This completes the proof that the subset of $A^n$ comprised of elements that,  as $\SL_2(\RR)$-matrices, are contained in powers of $\cC_j$ approaches zero measure as $n\to \infty$. Since there are finitely many $\cC_j$, we have proven our claim.
\end{proof}

It follows that Theorems  \ref{thm:prob1} to \ref{thm:prob5} yield the LLN, CLT, LDP, and LLT  for the distribution of geometric lengths when sampling from $\cS$-words of  length $n$ that are hyperbolic in $\Gamma$, with law $\mu^{*n}$. This is with the caveat that for Theorem~\ref{thm:prob1}, while the subset of $b\in B$ such that $g_n\ldots g_1$ is non-hyperbolic for all $n$ is not necessarily zero measure, the theorem statement relates to geometric length in the following way. The subset of $b\in B$ for which there exists $N_b$ such that $g_n\ldots g_1$ is non-hyperbolic for all $n>N_b$ has zero measure. This is because 
\[
\{b\in B: g_n\ldots g_1 \text{ is non-hyperbolic for all $n > j$} \}
\]
has  zero measure for any $j\ge 0$, so 
\[
\bigcup_{j\ge 0}\{b\in B : g_n\ldots g_1 \text{ is non-hyperbolic for all $n > j$}\}
\]
 also has zero measure. Thus, for almost all $b \in B$, there exists an increasing sequence of positive integers $\{n_{b,j}\}_{j\in \ZZ_{>0}}$ such that $g_{n_{b,j}}\cdots g_1$ is hyperbolic for all $j$, and 
 \[
 \lim_{j \to \infty} \frac{1}{n_{b,j}} \log\norm{g_{n_{b,j}} \cdots g_1} =\lambda_1.
 \]
 
 \section{Concluding remarks}
 
In this section, we describe some potential future directions of inquiry. An obvious such direction is Conjecture~\ref{conj:chas} for $S$ a pair of pants, which perhaps could be proven by a method inspired by the theory of random walks on a uniformizing Fuchsian group. Another line of inquiry is given by the Law of the Iterated Logarithm (LIL) for random walks on matrix groups.
\begin{theorem}[{{\cite[Theorem 1.8]{randomwalk}}}]\label{thm:prob3}
Suppose  hypotheses $(1)$ and $(2)$ hold. For $\be$-almost all $b \in B$, the set of limit points of 
\[
\left\{\frac{\log \norm{g_n \cdots g_1} -\ld_1 n}{\sqrt{2\Phi n \log \log n}} : n\in \ZZ_{>0}\right\}
\]
is $[-1,1]$.
\end{theorem}
\noindent Our proof of Lemma~\ref{lem:hyp} demonstrates that the conclusion of Theorem~\ref{thm:prob3} holds. However, this statement is about  norms of \emph{all}  matrices in a specified Fuchsian-group random-walk path $b\in B$, not just the norms of hyperbolic such matrices, for which the norm has a geometric meaning: the geometric length of the corresponding closed geodesic. In order to give our desired geometric interpretation of Theorem~\ref{thm:prob1}, we showed that for any $b\in B$,  we can take an increasing subsequence $\{n_{j}\}_{j\in \ZZ_{>0}}$ of $\ZZ_{>0}$ for which the matrices $g_{n_{j}}\cdots g_1$ are hyperbolic, so that the LLN-type statement held for this subsequence. We would also like a geometric interpretation of Theorem~\ref{thm:prob3}, which begs the question: can one take such a subsequence so that the set of limit points of 
\[
\left\{\frac{\log \norm{g_{n_j} \cdots g_1} -\ld_1 n_j}{\sqrt{2\Phi n_j \log \log n_j}}: j\in \ZZ_{>0}\right\}
\]
demonstrates a LIL-type behavior? We conjecture that the answer is yes, given that  LIL-type statements for subsequences are known~\cite{lil1,lil2,lil3,lil4}.

\section*{Acknowledgments}

\noindent  This exposition was supported by the National Science Foundation Graduate Research Fellowship Program (grant number DGE1745303). The author would like to thank Alex Eskin and Alex Wright for helpful conversations, and the anonymous referee for directing us to \cite[Corollary 14.16 and Theorem 14.22]{randomwalk}.

 \bibliographystyle{amsplain}
\bibliography{main}

\end{document}